\documentclass[10pt]{amsart}
\usepackage[nohug,small]{diagrams}

\author[D. Khosla]{Deepak Khosla}
\address{Department of Mathematics \\  University of Texas at Austin \\
  1 University   Station C1200 \\ Austin, Texas 78712 \\ USA}
\email{dkhosla@math.utexas.edu}

\date{\today}

\newtheorem{thm}{Theorem}[section]
\newtheorem{lemma}[thm]{Lemma}
\newtheorem{cor}[thm]{Corollary}
\newtheorem{prop}[thm]{Proposition}

\theoremstyle{definition}
\newtheorem{defn}[thm]{Definition}
\newtheorem{rmk}[thm]{Remark}

\newcommand{\oh}[1]{\mathcal{O}_{#1}}

\DeclareMathOperator{\Sym}{Sym}
\newcommand{\tensor}{\otimes}
\newcommand{\caniso}{\simeq}
\newcommand{\isom}{\cong}
\newcommand{\ds}{\oplus}
\newcommand{\Ds}{\bigoplus}
\newcommand{\dual}[1]{{#1}^{\vee}}

\newcommand{\coh}[2]{H^{#1}(#2)}

\newcommand{\incl}{\hookrightarrow}
\newcommand{\union}{\cup}
\newcommand{\Union}{\bigcup}

\newcommand{\Z}{\mathbf{Z}}
\newcommand{\Q}{\mathbf{Q}}
\newcommand{\C}{\mathbf{C}}

\newcommand{\Proj}{\mathbf{P}}
\newcommand{\Grass}{\mathbf{G}}
\newcommand{\hilb}[1]{\Hilb_{\Z}^{#1}\Proj^2}
\newcommand{\sch}[1]{\sigma_{#1}}
\newcommand{\ls}[1]{\lvert{#1}\rvert}

\DeclareMathOperator{\Pic}{Pic}
\DeclareMathOperator{\Bl}{Bl}
\DeclareMathOperator{\Hilb}{Hilb}
\DeclareMathOperator{\Spec}{Spec}

\DeclareMathOperator{\ch}{ch}
\DeclareMathOperator{\td}{td}

\newcommand{\sC}{\mathcal{C}}
\newcommand{\sL}{\mathcal{L}}
\newcommand{\sV}{\mathcal{V}}

\newcommand{\g}[2]{\mathfrak{g}_{#2}^{#1}}
\newcommand{\grd}{\g{r}{d}}
\newcommand{\Grd}{G^r_d}
\newcommand{\sG}[2]{\mathcal{G}^{#1}_{#2}}

\newcommand{\sgrd}{\mathcal{G}^r_d}

\newcommand{\scrd}{\mathcal{C}^r_d}
\newcommand{\gsix}{\mathcal{G}^6_{24}}

\newcommand{\gsixm}{\gsix}

\newcommand{\Mgb}{\overline{\mathcal{M}}_g}
\newcommand{\Mgn}{\mathcal{M}_{g,n}}
\newcommand{\Mgnb}{\overline{\mathcal{M}}_{g,n}}
\newcommand{\Mgob}{\overline{\mathcal{M}}_{g,1}}
\newcommand{\Mgnt}{\widetilde{\mathcal{M}}_{g,n}}
\newcommand{\Mg}{\mathcal{M}_g}

\newcommand{\Mgt}{\widetilde{\mathcal{M}}_{g}}
\newcommand{\Mgot}{\widetilde{\mathcal{M}}_{g,1}}
\newcommand{\Mttl}{\widetilde{\mathcal{M}}_{2,1}}
\newcommand{\Mogb}{\overline{\mathcal{M}}_{0,g}}
\newcommand{\mtl}{\widetilde{\mathcal{M}}}
\newcommand{\m}{\mathcal{M}}
\newcommand{\M}{\mathcal{M}}
\newcommand{\Mb}{\overline{\mathcal{M}}}
\newcommand{\mirr}{\m_{21}^{\mathrm{irr}}}
\newcommand{\Mgi}{\m_g^{\mathrm{irr}}}
\newcommand{\Mgoi}{\m_{g,1}^{\mathrm{irr}}}

\newcommand{\Ctl}{\widetilde{\mathcal{C}}_{g,1}}

\newcommand{\cb}{\overline{\mathcal{C}}}

\newcommand{\ac}{\alpha}
\newcommand{\bc}{\beta}
\newcommand{\cc}{\gamma}
\newcommand{\sg}{\sigma}
\newcommand{\Sg}{\Sigma}
\newcommand{\Sgt}{\overline{\Sigma}}
\newcommand{\lam}{\lambda}
\newcommand{\de}{\delta}
\newcommand{\De}{\Delta}
\newcommand{\dn}{\delta_0}
\newcommand{\di}{\delta_i}
\newcommand{\eps}{\epsilon}
\newcommand{\epsi}{\epsilon_i}
\newcommand{\can}{\omega}
\newcommand{\tE}{\widetilde{E}}

\newcommand{\mess}{\xi}

\newcommand{\pf}{\eta_*}
\newcommand{\pls}{\pi_*}

\newcommand{\SchS}{(\mathrm{Sch}/S)}
\newcommand{\Sets}{(\mathrm{Sets})}
\newcommand{\syrd}{\mathcal{Y}^r_d}
\newcommand{\shred}{\mathcal{H}^{r,e}_d}

\title[A push-forward formula when $\rho=0$]
{Tautological classes on moduli spaces of curves with linear
  series and a push-forward formula when $\rho=0$}

\begin{document}

\begin{abstract}
  We define tautological Chow classes on the moduli space $\sgrd$ of
  triples consisting of a curve $C$, a line bundle $L$ on $C$ of
  degree $d$, and a linear system $V$ on $L$ of dimension $r$. In the
  case where the forgetful morphism to $\Mgb$ has relative dimension
  zero, we describe the images of these classes in $A^1(\Mgb)$.
  As an application, we compute the (virtual) slopes of several
  different classes of divisors on $\Mgb$.
\end{abstract}

\maketitle{}

\tableofcontents{}

\section{Introduction}
\label{sec:introduction}

The cone of effective divisors on a projective variety plays an
important r\^ole in the understanding of its birational geometry. In
the case of the moduli space $\Mg$ of genus-$g$ curves, it has become
apparent that the most interesting effective divisor classes are those
that arise from the \emph{extrinsic} geometry of curves in projective
space. Indeed, in their pioneering work, Harris and Mumford
\cite{HarMum} considered divisors of curves admitting a degree-$d$
branched cover of $\Proj^1$, where $d=(g+1)/2$. More recently, work of
Cukierman \cite{cukierman}, Farkas-Popa \cite{FarkasPopa}, Khosla
\cite{thesis.old}, and Farkas \cite{Farkas.Syzygies} \cite{Farkas.Koszul} has
found effective divisor classes of smaller slope that those considered
by Mumford and Harris, and some of these classes have been used to
improve on the best known bounds on $n$ for which $\Mgn$ is of general
type for fixed $g$ \cite{Farkas.Koszul}.

Although all of these divisors are described by conditions on the space
of embeddings of a curve in projective space, the techniques used to
deal with them have been varied. In this paper, we put all of the
above calculations into a unified framework and lay the ground for
future work on the effective cone of $\Mgb$. Specifically we consider
the moduli stack $\sgrd(\Mg)$ of genus-$g$ curves together with a
$\grd$ (linear series). The set of $\C$-valued points consists of
triples $(C,L,V)$, where $C$ is a genus-$g$ curve, $L$ is a degree-$d$
line bundle on $C$, and $V\subset\coh{0}{L}$ is an $(r+1)$-dimensional
subspace. The forgetful morphism
\begin{equation*}
  \eta\colon\sgrd(\Mg) \to \Mg
\end{equation*}
is representable, proper, and generically smooth of
relative dimension
\begin{equation*}
  \rho(g,r,d) = g - (r+1)(g-d+r)
\end{equation*}
\cite{Kleiman.Laksov}, \cite{GriffithsHarrisBN}, \cite{EHcusp},
\cite{EH.Brill-Noether} \cite{Gieseker}, \cite{EHPetri},
\cite{Lazarsfeld}, \cite{Fulton.Lazarsfeld}. (When $\rho<0$, then
$\eta$ is not dominant.)

If $g$, $r$, and $d$ are chosen so that $\rho=-1$, then the image of
$\eta$ has a component of codimension $1$, and it is this divisor that
Eisenbud and Harris use to show that $\Mg$ is of general type when
$g\ge24$ and $g+1$ is composite \cite{EH.Kodaira}. The closure of this
divisor in the moduli space $\Mgi$ of irreducible nodal curves may
also be interpreted as the image of the virtual fundamental class of
$\sgrd(\Mgi)$ under the proper push-forward morphism
\begin{equation*}
  \eta_* \colon A_*(\sgrd(\Mgi)) \to A_*(\Mgi),
\end{equation*}
where $\sgrd(\Mgi)$ is a partial compactification of $\sgrd(\Mg)$
using torsion-free sheaves.

If we now choose $g$, $r$, and $d$ so that $\rho=0$, we are led to the
``second generation'' of effective divisors on $\Mg$. In this case,
the morphism $\eta\colon\sgrd(\Mg) \to \Mg$ is generically finite.
Since the work of Cukierman \cite{cukierman}, every interesting
effective divisor class on $\Mgb$ has been realized as the image under
$\eta$ of a divisor on $\sgrd(\Mg)$. For example, the K3 locus in
$\M_{10}$ \cite{cukierman}, which was the first counterexample
\cite{FarkasPopa} to the Harris-Morrison slope conjecture
\cite{HarMor}, can be interpreted as the image under $\eta$ of the
divisor in $\sG{4}{12}(\M_{10})$ of $\g{4}{12}$'s which do not lie on
a quadric.  Again, the class of its closure in
$\M_{10}^{\mathrm{irr}}$ may be realized as the proper push-forward of
the corresponding class in $\sG{4}{12}(\M_{10}^{\mathrm{irr}})$ under
the morphism
\begin{equation*}
  \eta_*\colon A_*(\sG{4}{12}(\M_{10}^{\mathrm{irr}})) \to A_*(\M_{10}^{\mathrm{irr}}).
\end{equation*}
This latter class, in turn, is easily computed to be
\begin{equation*}
  2\ac - \bc - 6\cc + \eta^*\lam,
\end{equation*}
where $\ac$, $\bc$, and $\cc$ are certain tautological classes
defined on $\sgrd(\Mgi)$. (See Sections~\ref{sec:tautological-classes}
and \ref{sec:syzygy-divisors}.)

In Section~\ref{sec:statement-theorem}, we introduce a partial
compactification $\sgrd(\Mgt)$ of $\sgrd(\Mg)$, which is proper over
an open substack $\Mgt$ of $\Mgb$ that contains $\Mg$ and whose
complement in $\Mgb$ has codimension 2. We define tautological virtual
codimension-1 Chow classes $\ac$, $\bc$, and $\cc$ on $\sgrd(\Mgt)$
and, in the case where $\rho=0$, compute their images under the proper
push-forward morphism
\begin{equation*}
  \eta_*\colon A_{3g-2}(\sgrd(\Mgt)) \to A_{3g-2}(\Mgt) = A^1(\Mgb).
\end{equation*}
This allows one to completely mechanically compute the slopes of all
of the divisor classes on $\Mgb$ that have thus far been studied. As
examples, in Section~\ref{sec:applications} we study syzygy divisors,
hypersurface divisors, Gieseker-Petri divisors, and secant plane
divisors.

In Section~\ref{sec:spec-famil-curv} we give the statements
of a series of calculuations over special families of stable curves.
These calculations assemble to give the main result. Finally
Section~\ref{sec:proofs-lemmas} is devoted to the proofs of the lemmas
stated in Section~\ref{sec:spec-famil-curv}.

This work was carried out for my doctoral thesis under the supervision
of Joe Harris. I would like to thank Ethan Cotterill, Gavril Farkas,
Johan de Jong, Martin Olsson, Brian Osserman, and Jason Starr for
helpful conversations.

\section{Statement of Theorem}
\label{sec:statement-theorem}

\subsection{A limit linear series moduli stack}
\label{sec:limit-linear-series}

\begin{defn}[\cite{Knudsen}]
  Let $S$ be any scheme, and let $g$ and $n$ be non-negative
  integers. An \emph{$n$-pointed stable curve of genus $g$ over $S$}
  is a proper flat morphism $\pi\colon X\to S$ together with sections
  $\sg_1\dotsc,\sg_n\colon S\to X$. Each geometric fiber $X_{\bar{s}}$
  must be a reduced, connected, $1$-dimensional scheme such that
  \begin{enumerate}
  \item $X_{\bar{s}}$ has only ordinary double points;
  \item $X_{\bar{s}}$ intersects the sections $\sg_1\dotsc,\sg_n$
    at distinct points $p_1,\dotsc p_n$ that lie on the smooth locus
    of $X_{\bar{s}}$;
  \item the line bundle $\can_{X_{\bar{s}}}(p_1+\dotsb+p_n)$ is ample;
  \item $\dim \coh{1}{\oh{X_{\bar{s}}}} = g$.
  \end{enumerate}
\end{defn}

\begin{thm}[\cite{DelMum},\cite{Knudsen}]
  \label{thm:Mgnb}
  Let $g$ and $n$ be non-negative integers such that $2g-2+n>0$.  The
  category $\Mgnb$ of families of $n$-pointed stable curves of genus
  $g$ is an irreducible Deligne-Mumford stack that is proper, smooth,
  and of finite type over $\Spec\Z$. 
\end{thm}

\begin{defn}[\cite{DelMum}]
  Let $k$ be an algebraically closed field, and let
  $X$ be an $n$-pointed stable curve of over $k$. The
  \emph{dual graph $\Gamma_X$ of $X$} is the following unoriented graph:
  \begin{enumerate}
  \item the set of vertices of $\Gamma_X$ is the set $\Gamma_X^0$ of
    irreducible components of $X$;
  \item the set of edges of $\Gamma_X$ is the set $\Gamma_X^1$ which
    is the union of the singular and marked points of $X$;
  \item an edge $x\in\Gamma^1_X$ has for extremities the irreducible
    components on which $x$ lies;
  \end{enumerate}
\end{defn}

\begin{defn}
  An $n$-pointed stable curve $X\to S$ is \emph{tree-like}
  if the dual graph of each geometric fiber is a tree.
\end{defn}

\begin{prop}[\cite{limit.linear.series}]
  The category $\Mgnt$ of families of tree-like curves is an open substack of
  $\Mgnb$, whose complement has codimension 2.
\end{prop}

\begin{defn}
  Let $\pi\colon X \to S$ be a smooth genus-$g$ curve. A $\grd$ on $X$
  is the data of a line bundle $L\to X$ of relative degree $d$
  together with a rank-$r$ vector subbundle $V$ of $\pls L$
  \cite[Definition 4.2]{Osserman}.
\end{defn}

\begin{prop}
  Let $\pi\colon X \to S$ be a smooth genus-$g$ curve. The \'etale
  sheafification of the functor
  \begin{equation*}
    \SchS \to \Sets
  \end{equation*}
  given by
  \begin{equation*}
    T \mapsto \{\text{\emph{$\grd$'s on $X_T\to T$}}\}
  \end{equation*}
  is represented by a scheme $\Grd(X/S)$, proper over $S$. If $Z$ is
  an irreducible component of $\Grd(X/S)$, then
  \begin{equation}
    \label{eq:1}
    \dim Z \ge \dim S + \rho(g,r,d)
  \end{equation}
\end{prop}

In this way, one can construct a Deligne-Mumford stack $\sgrd$,
representable and proper over $\Mg$.  In \cite{limit.linear.series},
Osserman and the author extend this construction to families of
tree-like stable curves. That is, for $X$ a tree-like stable curve
over $S$ of genus $g$, they construct an algebraic space $\Grd(X/S)$,
proper over $S$ and satisfying the same dimension lower
bound~(\ref{eq:1}). In addition, they prove that if $S=\Spec k$, where
$k$ is an algebraically closed field, then there is an open substack
of $\Grd(X/k)$ isomorphic to the space of refined limit linear series
on $X$.
In this way they construct a representable and proper Deligne-Mumford
stack $\sgrd$ over $\Mgt$ and hence over $\Mgnt$ by pulling back.

\subsection{Tautological Classes}
\label{sec:tautological-classes}

\begin{defn}
  Let $g$ and $n$ be as in Theorem~\ref{thm:Mgnb} and let $r$ and $d$
  be non-negative integers for which $\rho(g,r,d)=0$. Let
  $U\subset\Mgnt$ be the open substack over which
  $\eta\colon\sgrd\to\Mgnt$ is flat.
  Let $[\sgrd]\in A_{3g-3+n+\rho}(\sgrd)$ be the
  class of the closure of $\eta^{-1}(U)$. Similarly, let $[\scrd]\in
  A_{3g-2+n+\rho}(\scrd)$ be the class of the closure of
  $(\eta\circ\pi)^{-1}(U)$ in the universal curve $\pi\colon\scrd\to\sgrd$.
\end{defn}

In the following we work over $\Mgot$ in order to be able to
consistently define the universal line and vector bundles. There is a
universal pointed quasi-stable curve $\syrd\to\sgrd$ whose
stabilization is the universal stable curve $\scrd\to\sgrd$. Let
$\sg\colon\sgrd\to\syrd$ be the marked section. There is a universal
line bundle $\sL\to\syrd$ of relative degree $d$ together with a
trivialization $\sg^*\sL\isom\oh{\sgrd}$. On each geometric fiber,
$\sL$ has degree 1 on every exceptional curve, degree $d$ on the
pre-image of the component of stable curve containing the marked
point, and degree 0 on the pre-image of all other components of the
stable curve. There is a sub-bundle
\begin{equation*}
  \sV \incl \pls\sL
\end{equation*}
which, over each point in $\sgrd$, is equal to the aspect of the limit
linear $\grd$ on the component containing the marked point.

\begin{rmk}
  \label{rmk:pic-Mg1}
  By a theorem of Harer \cite{Harer}, for $g\ge3$,
  \begin{equation*}
    A_{3g-3}(\Mgot)_\Q = \Pic\Mgob\tensor\Q =
    \Q\lam\ds\Q\dn\ds\Q\de_1\ds\dotsb\ds\Q\de_{g-1}\ds\Q\psi
  \end{equation*}
  where $\lam$ and $\psi$ are the first Chern classes of the Hodge and
  tautological bundles respectively, $\dn$ is
  the divisor of irreducible nodal curves, and $\di$ is the divisor of
  unions of curves of genus $i$ and $g-i$, where the marked point lies
  on the component of genus $i$. 
\end{rmk}

\begin{defn}
  We define ``codimension-1'' cycle classes in $A_{3g-3+\rho}(\sgrd)$
  as follows.
  \begin{align*}
    \ac &= \pls \bigl(c_1(\sL)^2 \cap [\syrd]\bigr) \\
    \bc &= \pls \bigl(c_1(\sL)\cdot c_1(\omega)\cap[\syrd]\bigr) \\
    \cc &= c_1(\sV)\cap[\sgrd]\text{.}
  \end{align*}
\end{defn}

\subsection{A Push-Forward Formula}
\label{sec:push-forward-formula}

We now state our main result.

\begin{thm}\label{main-thm}
  Let $g \ge 1$, and $r,d \ge 0$, be integers for which
  \begin{equation*}
    \rho = g - (r+1)(g-d+r) = 0,
  \end{equation*}
  and consider the map
  \begin{equation*}
    \eta\colon \sgrd  \to \Mgot
  \end{equation*}
  If
  \begin{equation*}
    \pf \colon A_{3g-3}(\sgrd) \to A_{3g-3}(\Mgot)
  \end{equation*}
  is the proper push-forward morphism on corresponding Chow groups,
  then
  \begin{equation*}
    \begin{split}
      \frac{6(g-1)(g-2)}{dN}\pf\ac& = 
      6(gd - 2g^2 + 8d - 8g + 4) \lam \\
      & \quad + (2g^2 - gd + 3g - 4d - 2) \dn \\
      & \quad + 6\sum_{i=1}^{g-1} (g-i)(gd + 2ig - 2id - 2d) \di \\
      & \quad - 6d(g-2) \psi \text{,}
    \end{split}    
  \end{equation*}
  \begin{equation*}
    \begin{split}
      \frac{2(g-1)}{dN}\pf\bc
       = 12\lam - \dn 
        + 4 \sum_{i=1}^{g-1} (g-i)(g-i-1) \di  -  2(g-1)\psi\text{,}
    \end{split}
  \end{equation*}

  \begin{equation*}
    \begin{split}
      \frac{2(g-1)(g-2)}{N}\pf\cc
      & = \bigl[-(g+3)\mess + 5r(r+2)\bigr]\lam
      - d(r+1)(g-2)\psi \\
      & \quad + \frac{1}{6}\bigl[(g+1)\mess - 3r(r+2)\bigr]\dn \\
      & \quad + \sum_{i=1}^{g-1} (g-i) \bigl[i\mess +
      (g-i-2)r(r+2)\bigr]\di \text{,}
    \end{split}
  \end{equation*}
  where
  \begin{equation*}
    N = \frac{1!\cdot 2!\cdot 3! \cdots r!\cdot g!}
    {(g-d+r)!(g-d+r+1)!\cdots(g-d+2r)!}
  \end{equation*}
  and
  \begin{equation*}
    \mess = 3(g-1) + \frac{(r-1)(g+r+1)(3g-2d+r-3)}{g-d+2r+1} \text{.}
  \end{equation*}

\end{thm}

\section{Applications}
\label{sec:applications}

In this section, we will apply Theorem~\ref{main-thm} to various
classes of divisors on $\sgrd(\Mgi)$, where $g,r,d$ are chosen so that
$\rho(g,r,d)=0$. We can parameterize such choices using integers
$r,s\ge1$ and setting $g=(r+1)(s+1)$ and $d=r(s+2)$.

\subsection{The Gieseker-Petri Divisor}
\label{sec:gies-petri-divis}

Petri's theorem \cite{Gieseker} states that if $C$ is a general curve,
then for all line bundles $L$ on $C$, the natural map
\begin{equation*}
  \coh{0}{L}\tensor\coh{0}{K_C\tensor \dual{L}} \to \coh{0}{K_C}
\end{equation*}
is injective. This implies that if $g$, $r$, and $d$ are chosen so
that $\rho=0$, and $(C,L)$ is a $\grd$ on a general curve $C$, then
the natural map
\begin{equation*}
  V\tensor\coh{0}{K_C\tensor \dual{L}} \to \coh{0}{K_C}
\end{equation*}
is an isomorphism. Away from a subset of codimension greater than 1,
the sheaf $\pls(\can\tensor\dual{\sL})$ on $\sgrd(\Mgoi)$ is locally
free, and the degeneracy locus of the map of vector bundles
\begin{equation*}
  \sV\tensor \pls(\can\tensor\dual{\sL}) \to \pls(\can)
\end{equation*}
defines the \emph{Gieseker-Petri} divisor in $\sgrd$. We compute its
class as follows. By definition,
\begin{equation*}
  c_1(\pls(\can)) = \lam.
\end{equation*}
By Grothendieck-Riemann-Roch,
\begin{equation*}
  \begin{split}
    c_1(\pls\sL) - c_1(R^1\pls\sL)&=
    \pls\bigl[\ch(\sL)\cdot\td_{\scrd/\sgrd}\bigr]_1 \\
    &= \pls\biggl[\Bigl(1+c_1(\sL) + \frac{1}{2}c_1(\sL)^2\Bigr) \\
    &\qquad\qquad\qquad \cdot
    \Bigl(1-\frac{c_1(\omega)}{2} +
    \frac{c_1(\omega)^2+\kappa}{12}\Bigr)\biggr]_1 \\
    &= \frac{\ac}{2} - \frac{\bc}{2} + \lam.
  \end{split}
\end{equation*}
Thus,
\begin{align*}
  c_1\pls(\can\tensor\dual{\sL}) = - c_1(R^1\pls\sL)
  = \frac{\ac}{2} - \frac{\bc}{2} -\cc + \lam.
\end{align*}
It follows that our degeneracy locus in $\sgrd(\Mgoi)$ has class
\begin{equation*}
  \frac{r+1}{2}(-\ac+\bc) + (d+1-g)\cc -r\lam.
\end{equation*}
It is easy to see that the slope of the image divisor in $\Mgoi$ will
be symmetric in $r$ and $s$. Letting $x=(r+1)+(s+1)$ and
$y=(r+1)(s+1)$ and applying Theorem~\ref{main-thm}, we find that the
slope of the Gieseker-Petri divisor in $\Mgi$ is
\begin{equation*}
  \frac {6(2x+7{y}^{2}+7xy+x{y}^{2}+12y+{y}^{3})}{y \left( 4+y
 \right)  \left( y+1+x \right) }.
\end{equation*}

\subsection{Hypersurface Divisors}
\label{sec:hypers-divis}

Another natural substack in $\sgrd$ is the locus of $\grd$s which lie
on a hypersurface of degree $k$; that is, $\grd$'s $(L,V)$ for which
the restriction map
\begin{equation*}
  \Sym^k\sV\to\coh{0}{L^{\tensor k}}
\end{equation*}
has a non-trivial kernel. If $\rho=0$ and the above two vector spaces
have the same dimension, then this defines a virtual divisor in
$\sgrd$. Namely, we look at the degeneracy locus of the map of vector bundles
\begin{equation*}
  \Sym^k\sV \to \pls(\sL^{\tensor k}).
\end{equation*}
Note that if $d>g-1$, then $L^{k}$ is always non-special when $k\ge2$.

To compute the class of the degeneracy locus, observe first that
\begin{equation*}
  c_1(\Sym^k\sV) = \binom{r+k}{k-1}\cc.
\end{equation*}
By Grothendieck-Riemann-Roch,
\begin{equation*}
  \begin{split}
    c_1 (\pls\sL^{\tensor k})&=
    \pls\bigl[\ch(\sL^{\tensor k})\cdot\td_{\scrd/\sgrd}\bigr]_1 \\
    &= \pls\biggl[\Bigl(1+kc_1(\sL) + \frac{k^2}{2}c_1(\sL)^2\Bigr)\biggr. \\
    & \qquad \qquad \qquad \left.\cdot \Bigl(1-\frac{c_1(\omega)}{2} +
      \frac{c_1(\omega)^2+\kappa}{12}\Bigr)\right]_1 \\
    &= \frac{k^2}{2}\ac - \frac{k}{2}\bc + \lam.
  \end{split}
\end{equation*}
Applying Theorem~\ref{main-thm}, the image divisor in $\Mgi$ has slope
\begin{equation*}
  \frac{f(k,r,s)}{g(k,r,s)},
\end{equation*}
where $f$ and $g$ are (rather large) polynomials in $k$, $r$, and $s$,
which are, in turn, related by the identity
\begin{equation*}
  \binom{r+k}{k} = kr(s+2) - (r+1)(s+1) + 1.
\end{equation*}

\subsection{Syzygy Divisors}
\label{sec:syzygy-divisors}

Consider a basepoint-free $\grd$ $(C,L,V)$, so there is a map $f\colon
C\to\Proj\dual{V}$. On $\Proj\dual{V}$, we have the tautological sequence
\begin{equation*}
  0\to \oh{\Proj\dual{V}}(-1) \to \dual{V}\tensor\oh{\Proj\dual{V}}\to
  Q\to 0.
\end{equation*}
For any $i$, consider the restriction map to $C$:
\begin{equation*}
  \coh{1}{\wedge^i \dual{Q}\tensor\oh{\Proj\dual{V}}(2)}
  \to \coh{0}{\wedge^i f^*\dual{Q}\tensor L^{\tensor 2}}.
\end{equation*}
According to \cite[Proposition 2.5]{Farkas.Syzygies}, the map $f$
fails Green's property $(N_i)$ if and only if this restriction map
degenerates to a certain rank. In the case where the two vector spaces
have the same dimension, $f$ fails property $(N_i)$ exactly when the
restriction map is not an isomorphism.

To globalize this, consider the tautological sequence on
$u\colon\Proj\dual{\sV}\to\sgrd$:
\begin{equation*}
  0\to \oh{\Proj\dual{\sV}}(-1)\to u^*\dual{\sV}\to Q \to 0.
\end{equation*}
We will remove from $\sgrd$ the closed substack, isomorphic to
$\sC^r_{d-1}$, of $\grd$s with a basepoint, which has codimension
greater than 1. Then there is a morphism
$f\colon\syrd\to\Proj\dual{\sV}$ commuting with the projection to $\sgrd$.
\begin{diagram}
  \syrd & &\rTo^f & & \Proj\dual{\sV} \\
  &\rdTo <{\pi} & & \ldTo >u & \\
  & &\sgrd & & \\
\end{diagram}
Our restriction map now globalizes to
\begin{equation*}
  u_*\wedge^i \dual{Q}\tensor\oh{\Proj\dual{\sV}}(2)
  \to \pls\wedge^i f^*\dual{Q}\tensor \sL^{\tensor 2}.
\end{equation*}
Note also that all the higher direct images of these two bundles
vanish \cite[Proposition 2.1]{Farkas.Syzygies}.

Using the exact sequence
\begin{equation*}
  0\to \wedge^{i+1}\dual{Q}\tensor\oh{\Proj\dual{\sV}}(j) \to 
  \wedge^{i+1}u^*\sV\tensor\oh{\Proj\dual{\sV}}(j)\to
  \wedge^{i}\dual{Q}\tensor\oh{\Proj\dual{\sV}}(j+1)\to 0
\end{equation*}
for $j=0,1$, we have
\begin{equation*}
  \begin{split}
    \ch(u_*(\wedge^i\dual{Q}\tensor\oh{\Proj\dual{\sV}}(2)))
    &= \ch(\wedge^{i+1}\sV\tensor\sV) - \ch\wedge^{i+2}\sV \\
    &= \left[ \binom{r+1}{i+1}(r+1) - \binom{r+1}{i+2}\right] \\
    &\qquad+ \left[ \binom{r}{i}(r+1) + \binom{r+1}{i+1} - \binom{r}{i+1}
    \right]\cc + \dotsb\\
    &= (i+1)\binom{r+2}{i+2} + (r+2)\binom{r}{i}\cc + \dotsb
  \end{split}  
\end{equation*}

Applying Grothendieck-Riemann-Roch to $\pi$, we obtain that
\begin{align*} 
  \ch\pls(\wedge^if^*\dual{Q}\tensor\sL^{\tensor 2})&=
  \pls\bigl[f^*\ch\wedge^i\dual{Q}\exp(2c_1(\sL))\cdot\td_{\syrd/\sgrd}\bigr] \\
  &=\pls\bigl[f^*\ch\wedge^i\dual{Q}\exp(2c_1(\sL))\cdot\td_{\scrd/\sgrd}\bigr]. \\
\end{align*}
One computes
\begin{equation*}
  \begin{split}
    \ch\wedge^i\dual{Q} &= \binom{r}{i} +
    \binom{r-1}{i-1}(u^*\cc-\zeta) \\
    & \qquad - \binom{r-2}{i-2}\zeta u^*\cc +
    \frac{1}{2}\left[\binom{r-2}{i-2}-\binom{r-2}{i-1}\right]\zeta^2 +
    \dotsb,
  \end{split}
\end{equation*}
where $\zeta=c_1(\oh{\Proj\dual{\sV}}(1))$. It follows that
\begin{equation*}
  \begin{split}
    \ch\pls(\wedge^if^*\dual{Q}\tensor\sL^{\tensor 2}) &=
    (2d+1-g)\binom{r}{i} - d\binom{r-1}{i-1} \\
    & \quad + \left[ 2\binom{r}{i} - 2\binom{r-1}{i-1} +
      \frac{1}{2}\binom{r-2}{i-2} - \frac{1}{2}\binom{r-2}{i-1}
    \right] \ac \\
    & \quad + \left[ -\binom{r}{i} +
      \frac{1}{2}\binom{r-1}{i-1}\right] \bc  \\
    & \quad + \left[ (2d+1-g)\binom{r-1}{i-1} -
      d\binom{r-2}{i-2}\right] \cc \\
    & \quad + \binom{r}{i}\lam.
  \end{split}
  \end{equation*}
In order for the two vector bundles to have the same dimension,
therefore, we need that
\begin{equation*}
  (i+1)\binom{r+2}{i+2} = (2d+1-g)\binom{r}{i} - d\binom{r-1}{i-1},
\end{equation*}
which is achieved by setting $r=(i+2)s+2(i+1)$. The class of our
degeneracy locus in $\sgrd$ is
\begin{equation*}
  \begin{split}
    &\left[ 2\binom{r}{i} - 2\binom{r-1}{i-1} +
      \frac{1}{2}\binom{r-2}{i-2} - \frac{1}{2}\binom{r-2}{i-1}
    \right] \ac + \left[ -\binom{r}{i} +
      \frac{1}{2}\binom{r-1}{i-1}\right] \bc \\
    &\quad+ \left[ -
      (r+2)\binom{r}{i} + (2d+1-g)\binom{r-1}{i-1} -
      d\binom{r-2}{i-2}\right] \cc + \binom{r}{i}\lam.
  \end{split}
  \end{equation*}
In Section~\ref{sec:appendix} we prove that this locus is an actual
effective divisor when $i=0$ and $0\le s\le2$.
By Theorem~\ref{main-thm} it follows that the slope of the image locus
in $\Mgi$ is
\begin{equation}
  \label{eq:syzygy.slope}
  \frac{6f(i,t)}{t(i-2)g(i,t)}
\end{equation}
where
\begin{equation*}
  \begin{split}
    f(i,t)&=(24{i}^{2}+{i}^{4}+16+32i+8{i}^{3} ) {t} ^{7}+ (
    4{i}^{3}+{i}^{4}-16i-16 ) {t}^{6} \\
    &+ ( -13 {i}^{2}-7{i}^{3}+12-{i}^{4} ) {t}^{5}+ ( -{i}^{2}-14i-{
      i}^{4} -24-2{i}^{3} ) {t}^{4} \\
    &+ ( 2{i}^{2}+2{i}^{3}-6 i-4 )
    {t}^{3}+ ( 17{i}^{2}+{i}^{3}+50i+41 ) {t}^{ 2} \\
    &+ ( 7{i}^{2}+9+18i ) t+2+2i,
  \end{split}
  \end{equation*}
\begin{equation*}
  \begin{split}
    g(i,t)&=(12i+{i}^{3}+8+6{i}^{2} ) {t}^{6}+ ( -4
    i+{i}^{3}-8+2{i}^{2} ) {t}^{5} \\
    &+ ( -2-11i-{i}^{3}-7{ i}^{2} )
    {t}^{4}+ ( -{i}^{3}+5i ) {t}^{3} \\
    &+ ( 5 i+1+4{i}^{2} ) {t}^{2}+ (
    7i+11+{i}^{2} ) t+2+4 i,
  \end{split}
\end{equation*}
and $t=s+1$.

\subsection{Secant Plane Divisors}
\label{sec:secant-plane-divis}

Given a curve $C$ in $\Proj^r$, we can ask whether there is a
$k$-plane meeting $C$ in $e$ points---that is, an $e$-secant
$k$-plane. For example, an $m$-secant $0$-plane is just an $m$-fold
point of $C$.

If $(C,L,V)$ is the associated $\grd$, then this condition is
described by saying that there is an effective divisor $E$ on $C$ of
degree $e$ such that the restriction map
\begin{equation*}
  V\to\coh{0}{L_E}
\end{equation*}
has a kernel of dimension at least $r-k$.

To globalize this over $\Mgoi$, we consider the relative Hilbert scheme
of points on a family of curves. For a stable curve $X$ over $S$, the functor
\begin{equation*}
  \SchS\to\Sets
\end{equation*}
\begin{equation*}
  T\mapsto \{\text{subschemes $\Sigma\subset X_T$, finite of degree
    $e$ over $T$}\} 
\end{equation*}
is represented by a scheme $\Hilb^e(X/S)$, proper over $S$.

Now let $\shred=\Hilb^e(\scrd/\sgrd)$, and consider the projection
\begin{equation*}
  p\colon \shred\to\sgrd.
\end{equation*}
Let
\begin{equation*}
  \Sigma \subset \shred\times_{\sgrd} \scrd
\end{equation*}
be the universal subscheme. The is a natural map of vector bundles
\begin{equation*}
  p^*\sV \to (\pi_1)_*\pi_2^*\sL(\Sigma)
\end{equation*}
on $\shred$, and we are looking for the rank-$(k+1)$ locus of this map.
The class of the virtual degeneracy locus in $\sgrd$ is, therefore,
given by the Porteous formula as
\begin{equation}
  \label{eq:porteous}
  p_*\Delta_{e-k-1,r-k}\bigl(
  c((\pi_1)_*\pi_2^*\sL(\Sigma))/p^*c(\sV)\bigr).
\end{equation}
In order to get a locus of expected codimension one in $\sgrd$, we
need that
\begin{equation*}
  (e-k-1)(r-k)=e+1 
\end{equation*}
or $\rho(e,r-k-1,r)=-1$. Cotterill \cite[Theorem 1]{Cotterill} has
proved that, in this case, one obtains an actual effective divisor on
$\sgrd$. The computation~(\ref{eq:porteous}) can, in principle, be
carried out using the techniques in \cite{Ran}. Cotterill
\cite{Cotterill} has made some progress towards an explicit
calculation; there remains, however, some work to be done.  The final
answer will have the form
\begin{equation*}
  P_\ac \ac + P_\bc \bc + P_\cc \cc + P_\lam \lam + P_{\dn} \dn,
\end{equation*}
where the coeffcients are rational functions in $\Q(r,s,e,k)$.

\section{Special Families of Curves}
\label{sec:spec-famil-curv}

Our strategy for proving Theorem \ref{main-thm} will be to pull back
to various families of stable curves over which the space of linear
series is easier to analyze. In Section \ref{sec:definitions-families}
we define three families of pointed curves, and in Section
\ref{sec:comp-spec-famil} we compute $\pf$ for these special families.
In Section \ref{sec:pull-back-maps} we compute the pull-backs of
the standard divisor classes on $\Mgot$ to the base spaces of each of
our families. Assembling the results of these three sections, we
compute $\pf$ over the whole moduli space.

\subsection{Definitions of Families}
\label{sec:definitions-families}

\begin{defn}
  \label{def:mogb}
  Let $i\colon \Mogb \incl \Mgot$ be the family of marked stable curves defined
  by sending a $g$-pointed stable curve
  \begin{equation*}
    (C,p_1,\dotsc,p_g)
  \end{equation*}
  of genus 0 to the stable curve
  \begin{equation*}
    \bigl( C \union \Union_{i=1}^g E_i , p_0 \bigr)
  \end{equation*}
  of genus $g$, where $E_i$ are fixed non-isomorphic elliptic curves,
  attached to $C$ at the points $p_i$, and $p_0\in E_1$ is fixed as
  well.
\end{defn}

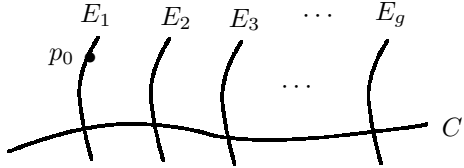
\begin{figure*}[h]
  \unitlength 1mm
\begin{picture}(62.57,24.11)(0,0)
\linethickness{0.3mm}
\qbezier(3.65,5.84)(12.78,9.41)(19.28,9.49)
\qbezier(19.28,9.49)(25.78,9.56)(30.87,7.95)
\qbezier(30.87,7.95)(37.34,6.77)(48.25,8.02)
\qbezier(48.25,8.02)(59.18,9.27)(59.16,9.45)
\linethickness{0.3mm}
\qbezier(15.53,21.03)(12.66,16.49)(13.08,12.68)
\qbezier(13.08,12.68)(13.49,8.88)(14.67,4.79)
\linethickness{0.3mm}
\qbezier(25.03,20.88)(22.16,16.34)(22.58,12.53)
\qbezier(22.58,12.53)(22.99,8.73)(24.17,4.64)
\linethickness{0.3mm}
\qbezier(34.54,20.42)(31.67,15.88)(32.08,12.08)
\qbezier(32.08,12.08)(32.49,8.28)(33.67,4.19)
\linethickness{0.3mm}
\qbezier(53.98,20.58)(51.11,16.04)(51.53,12.23)
\qbezier(51.53,12.23)(51.94,8.43)(53.12,4.34)
\put(62.57,9.15){\makebox(0,0)[cc]{$C$}}

\put(15.23,24.05){\makebox(0,0)[cc]{$E_1$}}

\put(25.9,23.65){\makebox(0,0)[cc]{$E_2$}}

\put(35,23.4){\makebox(0,0)[cc]{$E_3$}}

\put(54.49,24.11){\makebox(0,0)[cc]{$E_g$}}

\put(45.06,23.89){\makebox(0,0)[cc]{$\dotsb$}}

\put(42.2,14.43){\makebox(0,0)[cc]{$\dotsb$}}

\put(14.46,18.24){\makebox(0,0)[cc]{$\bullet$}}

\put(10.68,18.45){\makebox(0,0)[cc]{$p_0$}}

\end{picture}
  \caption{$i(C,p_1,\dotsc,p_g)$}
\end{figure*}

\begin{defn}
  \label{def:M21}
  Let $j\colon \Mttl \incl \Mgot$ be the family of curves defined by
  sending a marked curve $(C,p)$ to the marked stable curve
  \begin{equation*}
    (C \union C', p_0)
  \end{equation*}
  where $(C',p',p_0)$ is a fixed
  Brill-Noether-general
  curve in $\mtl_{g-2,2}$, attached nodally to $(C,p)$ at $p'$.
\end{defn}

\begin{figure*}[h]
%
%
%
%
%
%
\unitlength 1mm
\begin{picture}(58.24,13.11)(0,0)
\linethickness{0.3mm}
\qbezier(8.51,3.92)(17.95,3.67)(22.23,5.13)
\qbezier(22.23,5.13)(26.5,6.6)(30.61,12.83)
\linethickness{0.3mm}
\qbezier(19.41,13.11)(26.72,5.54)(33.67,3.95)
\qbezier(33.67,3.95)(40.62,2.35)(53.65,3.58)
\put(3.92,3.65){\makebox(0,0)[cc]{$C$}}

\put(58.24,3.38){\makebox(0,0)[cc]{$C'$}}

\put(45,3.11){\makebox(0,0)[cc]{$\bullet$}}

\put(45,6.35){\makebox(0,0)[cc]{$p_0$}}

\end{picture}
  \caption{$j(C,p)$}
\end{figure*}

\begin{defn}
  \label{def:marked-point}
  Fix Brill-Noether general curves
  \begin{align*}
    (C_1, p_1) & \in \m_{h,1} \\
    (C_2, p_2) & \in \m_{g-h,1} 
  \end{align*}
  and let $C = C_1 \union C_2$ be their nodal union along the $p_i$. Let
  $k_h\colon C_1 \incl \Mgot$ be the map sending $p\in C_1$ to the marked
  curve $(C,p)$.
\end{defn}

\subsection{Computations on the Special Families}
\label{sec:comp-spec-famil}

\begin{lemma}\label{lem:eta-mogb}
  For the family
  \begin{equation*}
    i\colon \Mogb \incl \Mgot
  \end{equation*}
  we have
  \begin{equation*}
    \pf\ac = \pf \bc = \pf\cc = 0
  \end{equation*}
\end{lemma}

\begin{lemma}\label{lem:eta-M21}
  For the family
  \begin{equation*}
    j \colon \Mttl \incl \Mgot
  \end{equation*}
  we have
  \begin{align*}
    \pf\ac
    &= \frac{2dN(d - 2g + 2)}{3(g-1)}(3\psi - \lam - \de_1)
    + \frac{dN}{g-1}(\lam + \de_1 - 4\psi) \\
    \pf\bc
    &= \frac{dN}{g-1}(\lam + \de_1 - 4\psi) \\
    \pf\cc
    &= \frac{-N\mess}{3(g-1)}(3\psi - \lam - \de_1) \text{,}
  \end{align*}
  where $N$ and $\mess$ are defined in the statement of Theorem \ref{main-thm}.
\end{lemma}

\begin{lemma}\label{lem:eta-marked-point}
  For the family
  \begin{equation*}
    k_h \colon C_1 \incl \Mgot
  \end{equation*}
  we have
  \begin{align*}
    \deg\pf\ac & = -d^2 N \\
    \deg\pf\bc & = -\bigl[2(g-h) - 1\bigr] dN \\
    \deg\pf\cc & = - \bigl[rh + \frac{1}{2}r(r+1)\bigr] N
  \end{align*}
\end{lemma}

\subsection{Pull-Back Maps on Divisors}
\label{sec:pull-back-maps}

\begin{lemma}
  \label{lem:pull-back-mogb}
  Let $\epsi$ be the class of the closure of the locus on $\Mogb$ of
  stable curves with two components, the component containing the
  first marked point having $i$ marked points.
  \begin{enumerate}
  \item The classes $\epsi$ are independent in $\coh{2}{\Mogb;\Q}$.
  \item
    For the family
    \begin{equation*}
      i\colon \Mogb \incl \Mgot
    \end{equation*}
    we have the following pull-back map on divisor classes.
    \begin{align*}
      i^* \lam & = i^*\psi = i^* \dn  = 0 \\
      i^* \di & = \epsi \qquad \textrm{for } i = 2,3,\ldots,g-2 \\
      i^*\de_1 & = -\sum_{i=2}^{g-2} \frac{(g-i)(g-i-1)}{(g-1)(g-2)}\epsi \\
      i^*\de_{g-1} & = -\sum_{i=2}^{g-2} \frac{(g-i)(i-1)}{g-2}\epsi \\
    \end{align*}
   \end{enumerate}
\end{lemma}

  

\begin{lemma}
  \label{lem:pull-back-M21}
  For the family
  \begin{equation*}
    j \colon \Mttl \incl \Mgot
  \end{equation*}
  we have the following pull-back map on divisor classes.
  \begin{align*}
    j^* \lam & = \lam & j^*\psi &= 0 \\
    j^* \dn & = \dn   & j^*\di &= 0 \quad i=1,2,\ldots,g-3 \\
    j^* \de_{g-2} & = -\psi & j^* \de_{g-1} & = \de_1
  \end{align*}
\end{lemma}

\begin{lemma}
  \label{pull-back-marked-point}
  For the family
  \begin{equation*}
    k_h\colon C_1 \incl \Mgot
  \end{equation*}
  we have the following pull-back map on divisor classes.
  \begin{align*}
    \deg k_h^*\lam & = 0 & \deg k_h^*\psi &= 2h -1 \\
    \deg k_h^*\de_h &= -1 & \deg k_h^*\de_{g-h} &= 1 \\
    \deg k_h^*\di &= 0 \quad i \ne h , g-h & & \\
  \end{align*}
\end{lemma}

\begin{proof}[Proof of Theorem \ref{main-thm}]
  Theorem \ref{main-thm} is now a consequence of the above lemmas.
  The main point is that the pull-backs of the classes
  $\pf\ac$, $\pf\bc$, and $\pf\cc$ to our special families coincide with the
  classes computed in Section~\ref{sec:comp-spec-famil}. For example,
  to see this for $j^*\pf\cc$, form the fiber the fiber square
  \begin{equation*}
    \begin{diagram}
      j^*\sgrd&\rTo^{j'}& \sgrd \\
      \dTo<{\eta'}& &\dTo>{\eta} \\
      \Mttl& \rTo_j& \Mgot\text{.}
    \end{diagram}
  \end{equation*}
  Notice that although $j$ is a regular embedding, $j'$ need not be.
  Nonetheless, according to Fulton \cite[Chapter 6]{Fulton.book}, there
  is a refined Gysin homomorphism
  \begin{equation*}
    j^! \colon A_k(\sgrd) \to A_{k-l}(j^*\sgrd)\text{,}
  \end{equation*}
  where $l$ is the codimension of $j$, which commutes with push-forward:
  \begin{equation*}
    \eta'_* j^! = j^*\pf\text{.}
  \end{equation*}
  We need to check that
  \begin{equation*}
    j^! c_1(\sV)\cap [\sgrd]
    = c_1({j'}^*\sV) \cap [j^*\sgrd]\text{.}
  \end{equation*}
  Since  
  \begin{equation*}
    j^! c_1(\sV)\cap [\sgrd]
    = c_1({j'}^*\sV) \cap j^![\sgrd]
  \end{equation*}
  \cite[Proposition~6.3]{Fulton.book}, it is enough to check that
  \begin{equation*}
    j^![\sgrd] = [j^*\sgrd]\text{.}
  \end{equation*}
  Generalizing the dimension upper bound in
  \cite[Corollary~5.9]{Osserman2} to the multi-component case 
  \cite{limit.linear.series}, we obtain
  \begin{equation*}
    \dim j^*\sgrd = \dim \Mttl \text{.}
  \end{equation*}
  This implies that the codimension of
  $j'$ is equal to that of $j$, so the normal cone of $j'$ is equal to
  the pull-back of the normal bundle of $j$, and the result follows.

  Now, for  example, to compute $\pf\cc$, write
  \begin{equation*}
    \pf\cc = a\lam - \sum_{i=0}^{g-1} b_i\di + c\psi
  \end{equation*}
  Our goal is to solve for $a,b_0,b_1,\dotsc,b_{g-1},c$. Using Lemmas
  \ref{lem:eta-marked-point} and \ref{pull-back-marked-point}, we may
  solve for $c$ and write $b_{g-i}$ in terms of $b_i$. From Lemmas
    \ref{lem:eta-mogb} and \ref{lem:pull-back-mogb}, we may further
    solve for $b_1,b_2,\dotsc,b_{g-2}$ in terms of $b_{g-1}$. It
    remains to determine $a$, $b_0$, and $b_{g-1}$. This is done by
    pulling back to $\Mttl$, which has Picard number 3, and using
    Lemmas \ref{lem:eta-M21} and \ref{lem:pull-back-M21}. The other
    push-forwards are computed similarly.
\end{proof}

\section{Proofs of Lemmas}
\label{sec:proofs-lemmas}

In this section, we give proofs of the lemmas stated in Sections
\ref{sec:comp-spec-famil} and \ref{sec:pull-back-maps}.

\begin{proof}[Proof of Lemma \ref{lem:eta-mogb}]
  \label{pf:eta-mogb}
  If $\cb_{0,g} \to \Mogb$ is the universal stable curve, then $i^*\Ctl$ is
  formed by attaching $\Mogb \times E_i$ to $\cb_{0,g}$ along the marked
  sections $\sigma_i \colon \Mogb \to \cb_{0,g}$. We have the
  following fiber square.
  \begin{equation*}
    \begin{diagram}
      i^*\scrd &\rTo& i^*\Ctl \\
      \dTo& &\dTo \\
      i^*\sgrd& \rTo& \Mogb
    \end{diagram}
  \end{equation*}

  By the Pl\"ucker formula for $\Proj^1$, given $[C]\in \Mogb$, a limit
  linear series on $i(C)$ must have the aspect
  \begin{equation*}
    (d-r-1)p_i + \ls{(r+1)p_i}
  \end{equation*}
  on each $E_i$. The line bundle $\sL \to i^*\scrd$ is, therefore,
  the pull-back from $i^*\Ctl$ of the bundle which is
  given by
  \begin{equation*}
    \pi_2^* \oh{E_1}{dp}
  \end{equation*}
  on $\Mogb \times E_1$ and is trivial on all other components. Thus
  $\ac=\bc=0$. The vector bundle $\sV \subset \pls \sL$ is trivial with
  fiber isomorphic to
  \begin{equation*}
    \coh{0}{\oh{E_1}{(r+1)p}}  
    \subset
    \coh{0}{\oh{E_1}{dp}}
  \end{equation*}
  so $\cc=0$ as well.
\end{proof}

Before proving Lemma \ref{lem:eta-M21} we state an elementary result
in Schubert calculus.

\begin{lemma}\cite[p. 266]{GriffithsHarrisBN}
  \label{lem:schubert}
  For integers $r$ and $d$ with $0\le r\le d$, let
  \begin{equation*}
    X = \Grass(r,\Proj^d)
  \end{equation*}
  be the Grassmannian of $r$-planes in $\Proj^d$. For integers
  \begin{equation*}
    0 \le b_0 \le b_1 \le \dotsb \le b_r \le d-r \text{,}
  \end{equation*}
  let $\sch{b}=\sch{b_r,\dotsc,b_0}$ be the corresponding Schubert cycle of
  codimension $\sum b_i$. Let $\zeta = \sch{1,1,\dotsc,1,0}$ be the
  special Schubert cycle of codimension $r$.
  If $k$ is an integer for which
  \begin{equation*}
    rk + \sum_{i=0}^r b_i  = \dim X = (r+1)(d-r) \text{,}
  \end{equation*}
  then
  \begin{equation*}
    \int_X \zeta^k \cdot \sch{b}
    = \frac{k!}{\prod_{i=0}^r (k-d+r + a_i)! }
    \prod_{0\le i<j\le r}(a_j - a_i) \text{,}
  \end{equation*}
  where $a_i = b_i + i$.
\end{lemma}

\begin{proof}[Proof of Lemma \ref{lem:eta-M21}]
  \label{pf:eta-M21}
  Since $\Mb_{2,1}$ is a smooth Deligne-Mumford stack, it is enough,
  by the moving lemma, to
  prove Lemma \ref{lem:eta-M21} for a family over a
  complete curve
  \begin{equation*}
    B \incl \Mttl
  \end{equation*}
  which intersects the boundary and Weierstrass divisors
  transversally.  If $\pi\colon\sC \to B$ is the universal stable
  genus-2 curve and $\sg\colon B\to \sC$ is the marked section, then
  $j^*\widetilde\sC_{g,1}$ is formed by attaching $\sC$ to $B \times
  C'$ along the marked section $\Sg=\sg(B) \subset \sC$.

  We begin by assuming that $B$ is disjoint from the closure of
  the Weierstrass locus $W$. In this case we claim that
  \begin{equation*}
    j^* \sgrd \to  B
  \end{equation*}
  is a trivial $N$-sheeted cover of the form $B\times X$, where is $X$
  a zero-dimensional scheme of length $N$.  Indeed, for any curve
  $(C,p)$ in $\Mttl \setminus W$ there are two (limit linear) $\grd$s
  on $C$ with maximum ramification at $p$; the vanishing sequences are
  \begin{align*}
    a_1 &= (d-r-2,d-r-1, \dotsc, d-4,d-3,d) \text{,} \\
    a_2 &= (d-r-2,d-r-1, \dotsc, d-4, d-2, d-1) \text{.}
  \end{align*}
  If $C$ is smooth, the two linear series are
  \begin{equation*}
    (d-r-2)p + \ls{(r+2)p}
  \end{equation*}
  and
  \begin{equation*}
    (d-r-2)p + \ls{rp + K_C} \text{.}
  \end{equation*}
  There are analogous series on nodal curves outside the closure of
  the Weierstrass locus. In the case of irreducible nodal curves, the
  sheaves are locally free.

  For each of the two $\grd$s on $C$ with maximum ramification at $p$,
  there are finitely many $\grd$s on $C'$ with compatible
  ramification. Specifically, there are
  \begin{equation*}
    \frac{(2g-2-d)N}{2(g-1)}
  \end{equation*}
  of type $a_1$ and
  \begin{equation*}
    \frac{dN}{2(g-1)} 
  \end{equation*}
  of type $a_2$, for total of $N$ limit linear series counted with
  multiplicity. Since $C'$ is
  fixed, the cover $j^*\sgrd \to B$ is a trivial $N$-sheeted cover.

  Consider a reduced sheet $B_1 \caniso B$ of type $a_1$. (We assume for
  simplicity that the sheet is reduced---the computation is the same
  in the general case.)  Then the universal
  line bundle $\sL$ on $j^*\sC^r_d$ is given as
  \begin{equation*}
    \sL \isom
    \begin{cases}
      \mathcal{O}_{\sC} & \text{on $\sC$} \\
      \pi_2^*L_1  & \text{on $B_1 \times C'$}
    \end{cases}
  \end{equation*}
  for some line bundle $L_1$ on $C'$ of degree $d$. It follows that
  $\ac=\bc=\cc=0$ on $B_1$.

  Next consider a sheet $B_2\caniso B$ of type $a_2$. Over $B_2 \times
  C'$ the universal line bundle $\sL$ is isomorphic to $\pi_2^* L_2$
  for some $L_2$ of degree $d$ on $C'$. It remains to determine $\sL$
  over $\sC$.
  Now $\omega_C(-2p)$ gives the correct line bundle for all $[C]\in
  B_2$; however, it has the wrong degrees on the components of the
  singular fibers. As our first approximation to $\sL$ on $\sC$ we take
  \begin{equation*}
    \omega_{\sC / B_2}( -2 \Sg)
  \end{equation*}
  Let $\De\subset\sC$ be the pull-back of
  the divisor on $\sC$ of curves of the form $C_1 \union C_2$, where
  the $C_i$ have genus one, and the marked points lie on different
  components. Then
  \begin{equation*}
    \omega_{\sC / B_2}( -2 \Sg + \De)
  \end{equation*}
  has the correct degree on the irreducible components on each
  fiber. It remains only to normalize our line bundle by pull-backs
  from the base $B_2$. In this case, $\sL|_\sC$ is required to be trivial
  along $\Sg$ since $\sL$ is a pull-back from $C'$ on the other
  component. If $\sg\colon B_2 \to \sC$  is the marked section, we
  let
  \begin{equation*}
    \Psi = \sg^* \omega_{\sC / B_2}
  \end{equation*}
  be the tautological line bundle on $B_2$. Then
  \begin{align*}
    \sg^* \oh{\sC}{\De} & \isom \mathcal{O}_{B_2} \\
    \sg^*\oh{\sC}{\Sg} & \isom \dual\Psi
  \end{align*}
  It follows that on $\sC$,
  \begin{equation*}
    \sL \isom
    \begin{cases}
      \omega_{\sC / B_2}( -2 \Sg + \De) \tensor \pi^*\Psi^{\tensor -3}
      &\text{on $\sC$} \\
      \pi_2^* L_2 &\text{on $B_2\times C'$.}
    \end{cases}
  \end{equation*}
  Thus, if we let
  \begin{align*}
    \omega &= c_1(\omega_{\sC / B_2}) \\
    \sg &= c_1 (\oh{\sC}{\Sg}) \\
    \de &= c_1 (\oh{\sC}{\De}) 
  \end{align*}
  on $\sC$ and let
  \begin{equation*}
    \psi = c_1(\Psi)
  \end{equation*}
  on $B_2$, then
  \begin{equation*}
    c_1(\sL) =
    \begin{cases}
      \omega - 2\sg + \de - 3\pi^*\psi  & \text{on $\sC$} \\
      d\pi_2^* p  & \text{on $B_2 \times C'$}
    \end{cases}
  \end{equation*}
  For the relative dualizing sheaf $\omega_{j^*\widetilde\sC_{g,1} /
    B_2}$, we have
  \begin{equation*}
    c_1(\omega_{j^*\sC / B_2}) =
    \begin{cases}
      \omega + \sg & \text{on $\sC$} \\
      (2(g-2) - 1)\pi_2^* p & \text{on $B_2 \times C'$.}
    \end{cases}
  \end{equation*}
  To compute the products of these classes on $j^*\sC^r_d$, recall the
  following formulas on $\pi\colon \overline{\sC}_{2,1} \to \Mb_{2,1}$
  \begin{align*}
    \pls \omega & =2 &
    \pls \delta & = 0 &
    \pls \sigma & = 1 \\
    \pls \omega^2 & = 12\lam - \dn - \de_1 &
    \pls\sg^2 & = -\psi &
    \pls\de^2 &= -\de_1 \\
    \pls (\de . \sg) & = 0 &
    \pls (\omega . \de) & = \de_1 &
    \pls (\sg . \omega) & = \psi \text{.}
  \end{align*}
  Then we compute
  \begin{align*}
    \ac &= \pls\bigl[c_1(\sL)^2\bigr] = 12\lam - \dn - 8\psi \\
    \bc &= \pls\bigl[c_1(\sL) \cdot c_1(\omega)\bigr] = 12\lam - \dn - 8\psi
  \end{align*}
  on $B_2$.
  Since the marked point lies on
  $C'$, $\sV$ is trivial on $B_2$, so $\cc=0$ on $B_2$.

  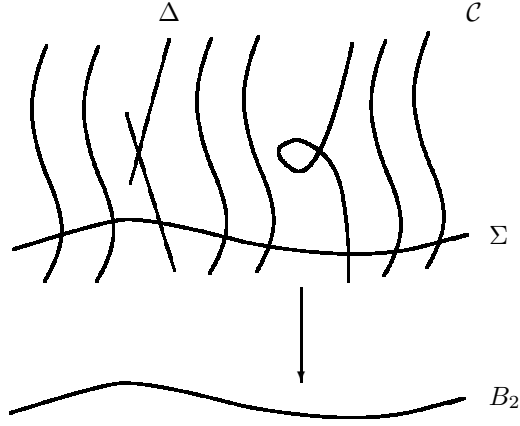
\begin{figure*}
    \unitlength 1mm
\begin{picture}(68.38,56.35)(0,0)
\linethickness{0.3mm}
\qbezier(57.57,53.51)(55.24,47.78)(55.71,43.62)
\qbezier(55.71,43.62)(56.18,39.45)(58.11,35.14)
\qbezier(58.11,35.14)(59.89,30.84)(59.56,28.14)
\qbezier(59.56,28.14)(59.24,25.43)(57.36,22.3)
\linethickness{0.3mm}
\multiput(17.97,33.38)(0.12,0.45){43}{\line(0,1){0.45}}
\linethickness{0.3mm}
\multiput(17.43,42.84)(0.12,-0.39){54}{\line(0,-1){0.39}}
\linethickness{0.3mm}
\qbezier(47.43,52.03)(43.13,32.56)(39.33,35.44)
\qbezier(39.33,35.44)(35.53,38.32)(40.14,39.32)
\qbezier(40.14,39.32)(45.18,37.85)(46.17,31.78)
\qbezier(46.17,31.78)(47.17,25.71)(46.89,20.54)
\linethickness{0.3mm}
\qbezier(2.03,2.97)(12.44,6.23)(15.96,6.84)
\qbezier(15.96,6.84)(19.49,7.44)(32.65,3.92)
\qbezier(32.65,3.92)(47.82,1.03)(56.64,3.41)
\qbezier(56.64,3.41)(65.46,5.79)(60.42,4.32)
\linethickness{0.3mm}
\qbezier(2.43,24.73)(12.84,27.99)(16.37,28.6)
\qbezier(16.37,28.6)(19.89,29.2)(33.06,25.68)
\qbezier(33.06,25.68)(48.23,22.79)(57.05,25.17)
\qbezier(57.05,25.17)(65.86,27.55)(60.82,26.08)
\linethickness{0.3mm}
\qbezier(51.86,52.46)(49.54,46.72)(50,42.56)
\qbezier(50,42.56)(50.47,38.4)(52.41,34.08)
\qbezier(52.41,34.08)(54.19,29.79)(53.86,27.08)
\qbezier(53.86,27.08)(53.53,24.37)(51.66,21.24)
\linethickness{0.3mm}
\qbezier(34.95,53.03)(32.62,47.29)(33.09,43.13)
\qbezier(33.09,43.13)(33.55,38.96)(35.49,34.65)
\qbezier(35.49,34.65)(37.27,30.35)(36.94,27.64)
\qbezier(36.94,27.64)(36.62,24.94)(34.74,21.81)
\linethickness{0.3mm}
\qbezier(28.7,52.78)(26.38,47.05)(26.84,42.89)
\qbezier(26.84,42.89)(27.31,38.72)(29.24,34.41)
\qbezier(29.24,34.41)(31.02,30.11)(30.7,27.4)
\qbezier(30.7,27.4)(30.37,24.7)(28.5,21.57)
\linethickness{0.3mm}
\qbezier(13.68,51.73)(11.35,45.99)(11.82,41.83)
\qbezier(11.82,41.83)(12.28,37.67)(14.22,33.35)
\qbezier(14.22,33.35)(16,29.06)(15.67,26.35)
\qbezier(15.67,26.35)(15.35,23.64)(13.47,20.51)
\linethickness{0.3mm}
\qbezier(6.76,51.76)(4.43,46.02)(4.9,41.86)
\qbezier(4.9,41.86)(5.36,37.69)(7.3,33.38)
\qbezier(7.3,33.38)(9.08,29.08)(8.75,26.37)
\qbezier(8.75,26.37)(8.43,23.67)(6.55,20.54)
\put(23.11,56.35){\makebox(0,0)[cc]{$\Delta$}}

\put(66.89,26.62){\makebox(0,0)[cc]{$\Sigma$}}

\put(63.65,56.35){\makebox(0,0)[cc]{$\sC$}}

\put(67.7,5){\makebox(0,0)[cc]{$B_2$}}

\linethickness{0.3mm}
\put(40.68,7.16){\line(0,1){12.43}}
\put(40.68,7.16){\vector(0,-1){0.12}}
\end{picture}
    \caption{The morphism $\sC\to B_2$.}
  \end{figure*}

  Finally we consider the case where $B$ (transversally) intersects
  the Weierstrass locus. In this case
  \begin{equation*}
    \eta\colon j^*\sgrd \to B
  \end{equation*}
  is the union of a trivial $N$-sheeted cover of $B$ and a 
  1-dimensional scheme lying over each point of the divisor $W$. It
  will suffice to compute $\ac$ and $\cc$ on
  \begin{equation*}
    \Grd ( j[C,p])
  \end{equation*}
  where $C$ is a smooth genus-2 curve, and $p\in C$ is a Weierstrass
  point. (Note that $\bc$ is automatically zero.)


  There is a single
  $\grd$ on $C$ with maximal ramification at $p$, namely
  \begin{equation*}
    (d-r-2)p + \ls{(r+2)p}\text{,}
  \end{equation*}
  which has vanishing sequence
  \begin{equation*}
    (d-r-2,d-r-1,\dotsc,d-4,d-2,d)\text{.}
  \end{equation*}
  We claim that we only need to consider components of $\Grd(C\union
  C')$ with this aspect on $C$. Indeed, any $\grd$ on $C$ with
  ramification 1 less at $p$ is still a subseries of $\ls{dp}$. There
  will be finitely many corresponding aspects on $C'$ so that, as
  before, $\ac=\bc=\cc=0$ on these components.

  It remains to consider the components of $\Grd(C\union C')$ where
  the aspect on $C'$ has ramification $(0,1,2,2,\dotsc,2)$ or more at
  $p'$. We are reduced to studying the one-dimensional scheme
  \begin{equation*}
    S=\Grd (C'; p',(0,1,2,2,\dotsc,2)) \text{.}
  \end{equation*}
  To simplify computations we specialize $C'$ to a curve which is the
  union of $\Proj^1$ with $g-2$ elliptic curves $E_1,\dotsc,E_{g-2}$
  attached at general points $p_1,\dotsc,p_{g-2}$,
  and where the marked point $p_0$ lies on $E_1$, and the point of attachment
  $p'$ lies on the $\Proj^1$.
  \begin{figure*}[h]
    \unitlength 1mm
\begin{picture}(42,28)(0,0)
\linethickness{0.3mm}
\put(4,4){\line(1,0){34}}
\linethickness{0.3mm}
\qbezier(10,24)(8.58,19.8)(8.86,16.76)
\qbezier(8.86,16.76)(9.15,13.71)(10.33,10.55)
\qbezier(10.33,10.55)(11.41,7.41)(11.21,5.43)
\qbezier(11.21,5.43)(11.02,3.45)(9.87,1.16)
\linethickness{0.3mm}
\qbezier(18,24)(16.58,19.8)(16.86,16.76)
\qbezier(16.86,16.76)(17.15,13.71)(18.33,10.55)
\qbezier(18.33,10.55)(19.41,7.41)(19.21,5.43)
\qbezier(19.21,5.43)(19.02,3.45)(17.87,1.16)
\linethickness{0.3mm}
\qbezier(32,24)(30.58,19.8)(30.86,16.76)
\qbezier(30.86,16.76)(31.15,13.71)(32.33,10.55)
\qbezier(32.33,10.55)(33.41,7.41)(33.21,5.43)
\qbezier(33.21,5.43)(33.02,3.45)(31.87,1.16)
\put(10,28){\makebox(0,0)[cc]{$E_1$}}

\put(18,28){\makebox(0,0)[cc]{$E_2$}}

\put(34,28){\makebox(0,0)[cc]{$E_{g-2}$}}

\put(24,16){\makebox(0,0)[cc]{$\dotsb$}}

\put(6,4){\makebox(0,0)[cc]{$\bullet$}}

\put(6,8){\makebox(0,0)[cc]{$p'$}}

\put(9.39,20){\makebox(0,0)[cc]{$\bullet$}}

\put(6,20){\makebox(0,0)[cc]{$p_0$}}

\put(42,4){\makebox(0,0)[cc]{$\Proj^1$}}

\put(26,28){\makebox(0,0)[cc]{$\dotsb$}}

\end{picture}
    \caption{The curve $C'$.}
  \end{figure*}
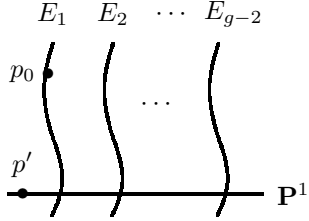
  There will be two types of components of $S$: those on which the
  aspects on the $E_i$ are maximally ramified at $p_i$, and those on
  which the aspect on one $E_i$ varies. Again, as in the proof of
  Lemma \ref{lem:eta-mogb}, we need only consider the latter case.

  Assume that for some $i$, the ramification at $p_i$ of the $\grd$ on
  $E_i$ is one less than maximal. There are two possibilities: either
  the series is of the form
  \begin{equation*}
    (d-r-1)p_i + \ls{rp_i+q} \qquad\text{for $q\in E_i$,}
  \end{equation*}
  which for $q\ne p_i$ imposes on the $\Proj^1$ the ramification condition
  \begin{equation*}
    (1,1,\dotsc,1) \text{,}
  \end{equation*}
  or the $\grd$ is a subseries of
  \begin{equation*}
    (d-r-2)p_i + \ls{(r+2)p_i}
  \end{equation*}
  containing
  \begin{equation*}
    (d-r)p_i + \ls{rp_i}\text{,}
  \end{equation*}
  which generically imposes on the $\Proj^1$ the ramification condition
  \begin{equation*}
    (0,1,1,\dotsc,1,2) \text{.}
  \end{equation*}
  In the first case the components are parameterized by $E_i$, and we
  compute that $\ac=-2$ on each such irreducible
  component, irrespective of whether $i=1$ or not. By
  Grothendieck-Riemann-Roch, $\cc=-1$ when $i=1$ and is zero otherwise. In
  the second case the $\grd$s are parameterized by a $\Proj^1$. Because
  the line bundle is constant, $\ac=0$. On each such $\Proj^1$,
  the vector bundle $\sV$
  may be viewed as the tautological bundle of rank $r+1$ on the
  Grassmannian of vector subspaces of a fixed vector space of
  dimension $r+2$ containing a subspace of dimension $r$. It follows
  that $\cc=-1$ on each $\Proj^1$.

  Let $X=\Grass(r,\Proj^d)$ be the Grassmannian of $r$-planes in $\Proj^d$. Let
  \begin{equation*}
    \zeta = \sch{1,1,\dotsc,1,0}
  \end{equation*}
  be the special Schubert cycle of codimension $r$.
  Collecting our calculations, we have that on $\Grd(C\union C')$,
  \begin{align*}
    \ac &= -2(g-2)\int_X
    \sch{2,2,\dotsc,2,1,0}\cdot\sch{1,1,\dotsc,1}\cdot\zeta^{g-3} \\ 
    &= -2(g-2)\int_X
    \sch{3,3,\dotsc,3,2,1}\cdot\zeta^{g-3} \text{,}
  \end{align*}
  and
  \begin{align*}
    \cc &= -\int_X \sch{2,2,\dotsc,2,1,0}\cdot
    (\sch{1,1,\dotsc,1} + \sch{2,1,1,\dotsc,1,0})\cdot
    \zeta^{g-3} \\
    &= -\int_X \sch{2,2,\dotsc,2,1,0}\cdot
    (\sch{1,0,0,\dotsc,0} \cdot \zeta)\cdot
    \zeta^{g-3} \\
    &= -\int_X
    (\sch{3,2,2,\dotsc,2,1,0}+\sch{2,2,\dotsc,2,0}+\sch{2,2,\dotsc,2,1,1})
    \cdot     \zeta^{g-2} \\ 
    &= -\int_X
    (\sch{3,2,2,\dotsc,2,1,0}+ \zeta^2)
    \cdot     \zeta^{g-2} \\ 
    &= -\int_X
    \sch{3,2,2,\dotsc,2,1,0} \cdot \zeta^{g-2}
    -\int_X \zeta^g \text{.}
  \end{align*}
  From Lemma \ref{lem:schubert} we compute,
  \begin{align*}
    \ac &= \frac{-2d(2g-2-d)N}{3(g-1)} \\
    \cc &= \frac{-\xi N}{3(g-1)} \text{.}
  \end{align*}
  Since the class of the Weierstrass locus in $\Mttl$ is $3\psi - \lam
  -\de_1$, the lemma follows.
\end{proof}

\begin{proof}[Proof of Lemma \ref{lem:eta-marked-point}]
  \label{pf:eta-marked-point}
  Because the curves $(C_i,p_i)$ are Brill-Noether general, $k_h^*
  \sgrd$ is a trivial $N$-sheeted cover of $C_1$ of the form
  $C_1\times X$, where $X$ is a zero-dimensional scheme of length $N$.
  Fix a sheet $G\isom C_1$ in $k_h^*\sgrd$; this choice corresponds to
  aspects
  \begin{equation*}
    V_i \subset \coh{0}{C_i, L_i}
  \end{equation*}
  where $L_i$ are degree-$d$ line bundles on $C_i$.
  If $(a_0,a_1,\dotsc,a_r)$ is the vanishing sequence of $V_1$ at
  $p_1$, then we know that
  \begin{align*}
    0 &= \rho(h,r,d) - \sum_{i=0}^r (a_i-i) \\
      &= (r+1)(d-r) -hr  - \sum_{i=0}^r a_i + \frac{1}{2}r(r+1)\text{,}
  \end{align*}
  so
  \begin{equation}
    \label{sum.a_i}
        \sum_{i=0}^r a_i = (r+1)d - \frac{1}{2}r(r+1) - hr \text{.}
  \end{equation}
  Let $\sC_1$ be the blow-up of $C_1 \times C_1$ at $(p_1,p_1)$, $E$
  the exceptional divisor, and $e$ its first Chern class. We may
  construct the universal curve $k_h^*\widetilde\sC_{g,1} \to C_1$ by
  attaching $C_1\times C_2$ to $\sC_1$ along $C_1 \times \{p_2\}$ and
  the proper transform of $C_1\times \{p_1\}$.  Over the sheet $G$,
  the universal line bundle $\sL$ on $k_h^*\sC^r_d$ is
  \begin{equation*}
    \pi_2^*L_1 \tensor \oh{\sC_1}{-dE}
    \tensor \pi_1^* \dual{L}_1(dp_1)
  \end{equation*}
  on $\sC_1$ and
  \begin{equation*}
    \pi_2^* L_2 (-dp_2) \tensor \pi_1^* \dual{L}_1
  \end{equation*}
  on $C_1 \times C_2$. Thus
  \begin{equation*}
    c_1(\sL) =
    \begin{cases}
      d \pi_2^* p - de & \text{on $\sC_1$} \\
      - d \pi_1^* p & \text{on $C_1 \times C_2$} \text{.}
    \end{cases}
  \end{equation*}
  The relative dualizing sheaf $\omega_{k_h^*\widetilde\sC_{g,1} /
    C_1}$ is isomorphic to
  \begin{equation*}
     \pi_2^* \omega_{C_1}\tensor \oh{\sC}{E} \tensor \pi_1^* \oh{C_1}{-p_1}
  \end{equation*}
  on $\sC_1$ and
  \begin{equation*}
    \pi_2^* \omega_{C_2}(p_2)
  \end{equation*}
  on $C_1 \times C_2$. We have
  \begin{equation*}
    c_1(\omega) =
    \begin{cases}
      - \pi_1^* p + (2h-2)\pi_2^* p + e & \text{on $\sC_1$} \\
      (2(g-h) - 1)\pi_2^* p & \text{on $C_1 \times C_2$} \text{.}
    \end{cases}
  \end{equation*}
  Thus, on $G$,
  \begin{align*}
    \deg\ac &= c_1(\sL)^2 = -d^2 \\
    \deg\bc &= c_1(\sL) \cdot c_1(\omega) = -d\bigl[2(g-j) - 1\bigr] \text{.}
  \end{align*}
  The formulas for $\pf\ac$ and $\pf\bc$ now follow.

  To calculate $\cc$ on $G$, notice that it suffices to
  compute $c_1(\sV')$, where
  \begin{equation*}
    \sV' = \sV \tensor L_1 (-d p_1)
  \end{equation*}
  is a sub-bundle of
  \begin{equation*}
    {\pi_1}_* (\pi_2^*L_1 (-dE)) \text{.}
  \end{equation*}
  We claim there is a bundle isomorphism
  \begin{equation*}
    \Ds_{j=0}^r\oh{C_1}{(a_j-d)p_1} \xrightarrow{\isom} \sV'
  \end{equation*}
  To describe the map, pick a basis $(\sg_0,\sg_1,\ldots,\sg_r)$ of
  $V_1\subset H^0(L_1)$ with $\sg_i$ vanishing to order $a_i$ at
  $p_1$. Given local sections $\tau_i$ of $\oh{C_1}{(a_i - d)p_1}$, let the
  image of $(\tau_0,\tau_1,\dotsc,\tau_r)$ be the section
  \begin{equation*}
    \sum_{i=0}^r \sg_i \tau_i
  \end{equation*}
  of $\sV'$. This is clearly an isomorphism away from $p_1$ and is
  checked to be an isomorphism over $p_1$ as well. 
  Using \eqref{sum.a_i}, we have that on $G$,
  \begin{equation*}
    \begin{split}
      \deg\cc = \deg\sV' &= \sum_{i=0}^r (a_i-d) \\
      &= -\frac{1}{2}r(r+1) - rh
    \end{split}
  \end{equation*}
  which finishes the proof of the lemma.
\end{proof}

  

\begin{proof}[Proof of Lemma \ref{lem:pull-back-mogb}]
  \label{pf:pull-back-mogb}
  To prove the independence of the $\epsi$, consider the curves
  \begin{equation*}
    B_j \incl \Mogb
  \end{equation*}
  for $j=2,3,\dotsc,g-3$ given by taking a fixed stable curve in
  $\eps_j$ and moving a marked point on the component with $g-j$
  marked points. Let $B_1\incl\Mogb$ be the curve given by moving the
  first marked point along a fixed smooth curve.
  The intersection matrix
  \begin{equation*}
    (\epsi \cdot B_j)
  \end{equation*}
  is
  \begin{equation*}
    \begin{pmatrix}
      g-1 &0&0&0&\cdots&0&0&0 \\
      -1&1&0&0&\cdots&0&0&g-3 \\
      0&-1&1&0&\cdots&0&0&g-4 \\
      \vdots& & & &\ddots& & & \vdots \\
      0&0&0&0&\cdots&-1&1&3 \\
      0&0&0&0&\cdots&0&-1&2 
    \end{pmatrix}
  \end{equation*}
  where the rows correspond to the $B_j$ for $j=1,2,\dotsc,g-3$, and
  the columns to the $\epsi$ for $i=2,3,\dotsc,g-2$. Since this matrix
  is non-singular, the first part of the lemma follows.

  To derive the formula for the pull-back, we follow Harris and
  Morrison \cite[Section 6.F]{HarMor}. Let $B$ be a smooth projective
  curve, $\pi\colon\sC \to B$ a 1-parameter family of curves in
  $\Mogb$ transverse to the boundary strata. Then $\pi$ has smooth
  total space, and the fibers of $\pi$ have at most two irreducible
  components. Let
  \begin{equation*}
    \sg_i \colon B\to\sC
  \end{equation*}
  be the marked sections. Denote by $\Sg_i$ the image curve $\sg_i(B)$
  in $\sC$. Then on $B$,
  \begin{align*}
    \de_1 &= \Sg_1^2 \\
    \de_{g-1} &= \sum_{j=2}^g \Sg_j^2 \text{,}
  \end{align*}
  where we are using $D^2$ to denote $\pls(D^2)$ for a divisor $D$ on $\sC$.
  We now contract the component of each reducible fiber which meets the
  section $\Sg_1$. If $\Sgt_j$ is the image of $\Sg_j$ under this
  contraction, then we have
  \begin{align*}
    \Sg_1^2 &= \Sgt_1^2 - \sum_{i=2}^{g-2}\epsi \\
    \sum_{j=2}^g \Sg_j^2 &= \sum_{j=2}^g\Sgt^2_j - \sum_{i=2}^{g-2}(i-1)\epsi
  \end{align*}
  The $\Sgt_j$ are sections of a $\Proj^1$-bundle, so
  \begin{equation*}
    \begin{split}
      0 &= (\Sgt_j - \Sgt_k)^2 
        = \Sgt_j^2 + \Sgt_k^2 - 2\Sgt_j\cdot\Sgt_k
    \end{split}
  \end{equation*}
  Thus
  \begin{equation*}
    \begin{split}
      (g-2)\sum_{j=2}^g\Sgt_j^2
      &= \sum_{2\le j,k\le g} (\Sgt_j^2 + \Sgt_k^2) 
       = 2\sum_{2\le j,k\le g} \Sgt_j\cdot\Sgt_k \\
      &= 2\sum_{i=2}^{g-2}\binom{i-1}{2}\epsi \text{.}
    \end{split}
  \end{equation*}
  It follows that
  \begin{equation*}
    \begin{split}
      \de_{g-1}
      &=
      \sum_{i=2}^{g-2}\left[\frac{(i-1)(i-2)}{g-2}-(i-1)\right]\epsi
      = \sum_{i=2}^{g-2}\frac{(i-1)(i-g)}{g-2}\epsi
    \end{split}
  \end{equation*}
  Similarly, we can show that
  \begin{equation*}
    \de_1 + \de_{g-1}
    = \sum_{i=2}^{g-2}\frac{i(i-g)}{g-1}\epsi
  \end{equation*}
  so the formula for $\de_1$ follows as well.
\end{proof}

The proofs of Lemmas \ref{lem:pull-back-M21} and
\ref{pull-back-marked-point} are straightforward, so we omit them.

\section{Appendix}
\label{sec:appendix}

In this section, we prove that the locus in
Section~\ref{sec:syzygy-divisors} defined for $i=0$ and $s=2$ is, in
fact, a divisor in $\gsix(\M_{21})$. The case $s=1$ is similar, and
$s=0$ is clear.

We first establish the following result.

\begin{lemma}
  \label{lem:irred}
  The space $\gsix(\M_{21})$ is irreducible.
\end{lemma}

\begin{proof}
  Note that a $\g{6}{24}$ is residual to a $\g{2}{16}$; that is
  \begin{equation*}
    L\in W^6_{24}(C) \iff K_C\tensor L^* \in W^2_{16}(C)
  \end{equation*}
  for any smooth curve $C$ of genus 21.
  Thus there is a dominant rational map
  \begin{equation*}
    V_{16,21} \longrightarrow \gsix
  \end{equation*}
  from the Severi variety of irreducible plane curves of degree 16 and
  genus 21. Since $V_{16,21}$ is irreducible \cite{SeveriProb} and
  maps dominantly to $\gsix$, so $\gsix$ is irreducible.
\end{proof}

\begin{prop}
  \label{prop:E-tilde-divisor}
  The substack $\tE$ of $\gsix(\M_{21})$ defined in
  Section~\ref{sec:syzygy-divisors} has codimension 1.
\end{prop}

\begin{proof}
  Since $\gsix(\m_{21})$ is irreducible, it suffices to exhibit a
  smooth curve with a $\g{6}{24}$ not lying on a quadric. Let
  \begin{equation*}
    S = \Bl_{21}\Proj^2
  \end{equation*}
  be the blow-up of $\Proj^2$ at 21 general points, and consider the
  linear system
  \begin{equation*}
    \nu = \Bigl\lvert 13H - 2\sum_{j=1}^9E_j - 3\sum_{k=10}^{21}E_k \Bigr\rvert
  \end{equation*}
  on $S$, where $H$ is the hyperplane class and $E_i$ are the
  exceptional divisors.  A calculation using Macaulay 2 (see
  Proposition~\ref{prop:macaulay}) shows that a general member $C$ of
  $\nu$ is irreducible and smooth of genus 21. The series
  \begin{equation*}
    \Bigl\lvert 6H - \sum_{i=1}^{21}E_i\Bigr\rvert
  \end{equation*}
  embeds $S$ in $\Proj^6$ as the rank-2 locus of general $3\times6$
  matrix of linear forms \cite[Section 20.4]{Eisenbud.comm.alg}.
  The ideal of $S$ is therefore generated by cubics, so $S$ does not
  lie on quadric. It follows that $C$, which
  embeds in $\Proj^6$ in degree 24,
  does not lie on a quadric.
 
\end{proof}

\begin{prop}
  \label{prop:macaulay}
  Let $\Sigma$ be a set of 21 general points in $\Proj^2$ and let $S =
  \Bl_{\Sigma}\Proj^2$ be the blow-up of $\Proj^2$ at $\Sigma$. If $H$
  is the line class on $S$ and $E_1,\dotsc,E_{21}$ are the exceptional
  divisors, then the linear system
  \begin{equation*}
     \Bigl\lvert 13H - 2\sum_{j=1}^9E_j - 3\sum_{k=10}^{21}E_k \Bigr\rvert
  \end{equation*}
  on $S$ contains a smooth connected curve.
\end{prop}

\begin{proof}
  We begin by showing that it is enough to exhibit a single set of 21
  points over a finite field for which the above statement is true.

  Let $\hilb{k}$ be the Hilbert scheme of $k$ points in $\Proj^2$, and
  let
  \begin{equation*}
    \Sigma_k \subset \hilb{k}\times\Proj^2
  \end{equation*}
  be the universal
  subscheme. Let
  \begin{equation*}
    B\subset\hilb{9}\times\hilb{12}
  \end{equation*}
  be the irreducible
  open subset over which the composition
  \begin{equation*}
    \begin{diagram}
      \Sigma = \pi_1^{-1}\Sigma_9 \union \pi_2^{-1}\Sigma_{12}
    &\rTo& \hilb{9}\times\hilb{12} \times \Proj^2 \\
    & &\dTo \\
     & & \hilb{9}\times\hilb{12} 
    \end{diagram}
  \end{equation*}
  is \'etale, where $\pi_1$ and $\pi_2$ are the obvious projections. Let
  \begin{equation*}
    \pi\colon S = \Bl_\Sigma\Proj^2_B \to B
  \end{equation*}
  be the smooth surface over B whose fibers are blow-ups of $\Proj^2$
  at 21 distinct points. If $E_9$ and $E_{12}$ are the exceptional
  divisors, let
  \begin{equation*}
    \sL = \oh{S}{13H - 2E_9 -3E_{12}}\text{.}
  \end{equation*}
  We may further restrict $B$ to an open over which $\pi_*\sL$ is
  locally free of rank at least 6.

  If
  \begin{equation*}
    \sC \subset \Proj\pi_*\sL \times_B S
  \end{equation*}
  is the universal section, then the projection
  \begin{equation*}
    \sC \to \Proj\pi_*\sL
  \end{equation*}
  is flat, so it suffices to find a single smooth fiber in order to
  conclude that the general fiber is smooth. To this end we use
  Macaulay 2 \cite{Macaulay2} and work over a finite field.

\begin{verbatim}

i1 : S = ZZ/137[x,y,z];

\end{verbatim}
  Following Shreyer and Tonoli \cite{Schreyer.Tonoli}, we realize our
  points in $\Proj^2$ as a determinental subscheme.
\begin{verbatim}

i2 : randomPlanePoints = (delta,R) -> (
       k:=ceiling((-3+sqrt(9.0+8*delta))/2);
       eps:=delta-binomial(k+1,2);
       if k-2*eps>=0
       then minors(k-eps,
         random(R^(k+1-eps),R^{k-2*eps:-1,eps:-2}))
       else minors(eps,
         random(R^{k+1-eps:0,2*eps-k:-1},R^{eps:-2})));

i3 : distinctPoints = (J) -> (
       singJ = minors(2, jacobian J) + J;
       codim singJ == 3);

\end{verbatim}
  Let $\Sigma_9$ and $\Sigma_{12}$ be our subsets of 9 and 12 points,
  respectively.
\begin{verbatim}

i4 : Sigma9 = randomPlanePoints(9,S);

o4 : Ideal of S

i5 : Sigma12 = randomPlanePoints(12,S);

o5 : Ideal of S

i6 : (distinctPoints Sigma9, distinctPoints Sigma12)

o6 = (true, true)

o6 : Sequence

\end{verbatim}
  Their union is $\Sigma$.
\begin{verbatim}

i7 : Sigma = intersect(Sigma9, Sigma12);

o7 : Ideal of S

i8 : degree Sigma

o8 = 21

\end{verbatim}
  Next we construct the 0-dimensional subscheme $\Gamma$ whose ideal
  consists of curves double through points of $\Sigma_9$ and triple
  through points of $\Sigma_{12}$.
\begin{verbatim}

i9 : Gamma = saturate intersect(Sigma9^2, Sigma12^3);

o9 : Ideal of S

\end{verbatim}
  Let us check that $\Gamma$ imposes the expected number of conditions
  ($9\cdot 3+12\cdot 6=99$) on curves of degree 13.
\begin{verbatim}

i10 : hilbertFunction (13, Gamma)

o10 = 99

\end{verbatim}
  Pick a random curve $C$ of degree 13 in the ideal of $\Gamma$.
\begin{verbatim}

i11 : C = ideal (gens Gamma 
                  * random(source gens Gamma, S^{-13}));

o11 : Ideal of S

\end{verbatim}
  We check that $C$ is irreducible.
\begin{verbatim}

i12 : # decompose C

o12 = 1

\end{verbatim}
  \newcommand{\Csing}{C_{\mathrm{sing}}}%
  To check smoothness, let $\Csing$ be the singular locus of $C$.
\begin{verbatim}

i13 : Csing = (ideal jacobian C) + C;

o13 : Ideal of S

i14 : codim Csing

o14 = 2

\end{verbatim}

  A double point will contribute 1 to the degree of $\Csing$ if it is
  transverse and more otherwise. Similarly, a triple point will
  contribute 4 to the degree of $\Csing$ if it is transverse and more
  otherwise. So for $C$ to be smooth in the blow-up, we must have that
  \begin{equation*}
    \deg\Csing = 9 + 4\cdot 12 = 57
  \end{equation*}
\begin{verbatim}

i15 : degree Csing

o15 = 57

\end{verbatim}
\end{proof}

\begin{defn}
  \label{def:E}
  Let $E$ be the effective codimension-1 Chow cycle which is the image
  of $\tE$ under the map
  \begin{equation*}
    \eta\colon \gsixm \to \mirr
  \end{equation*}
\end{defn}

\begin{prop}
  \label{prop:class-E}
  The class of $E\subset \mirr$ is given as
  \begin{equation*}
    [E] = 2459\lam - 377\dn \text{.}
  \end{equation*}
\end{prop}

\begin{proof}
  Applying Equation~(\ref{eq:syzygy.slope}) from
  Section~\ref{sec:syzygy-divisors}, 
  \begin{equation*}
    \begin{split}
      [E] = \pf[\tE] &= \pf(2\ac - \bc + \lam  - 8\cc) \\
      &= \frac{2459N}{95}\lam - \frac{377N}{95}\dn \text{,}
    \end{split}
  \end{equation*}
  where $N$ is the degree of $\eta$.
\end{proof}

\begin{cor}
  The slope conjecture is false in genus 21.
\end{cor}

\begin{proof}
  Since
  \begin{equation*}
    \frac{2459}{377} < 6 + \frac{12}{22},
  \end{equation*}
  this is an immediate consequence of \cite[Corollary 1.2]{FarkasPopa}.
\end{proof}

\bibliographystyle{hamsplain}
\bibliography{math}

\end{document}